\date{February 24, 2020}

\documentclass[10pt]{amsart}
\usepackage{latexsym,amsmath,amsfonts,amscd,amssymb,mathtools}
\usepackage{graphics}
\usepackage{xcolor}
\usepackage{tikz}
\newtheorem{theorem}{Theorem}[section]
\newtheorem{lemma}[theorem]{Lemma}
\newtheorem{proposition}[theorem]{Proposition}
\newtheorem{corollary}[theorem]{Corollary}

\newtheorem{definition}[theorem]{Definition}

\theoremstyle{remark}

\newcommand{\Res}{\text{Res}}

\newcommand{\Mon}{\operatorname{Mon}}
\newcommand{\BMon}{\operatorname{B-Mon}}
\newcommand{\Ker}{\operatorname{Ker}}

\newcommand{\supp}{\operatorname{supp}}

\newcommand{\cM}{{\mathcal M}}

\newcommand{\cP}{{\mathcal P}}
\newcommand{\cR}{{\mathcal R}}
\newcommand{\cS}{{\mathcal S}}

\newcommand{\CC}{{\mathbb C}}

\newcommand{\RR}{{\mathbb R}}
\newcommand{\TT}{{\mathbb T}}

\newcommand{\ZZ}{{\mathbb Z}}

\renewcommand{\Re}{\operatorname{Re}}
\renewcommand{\Im}{\operatorname{Im}}

\newcommand{\g}{\gamma}

\newcommand{\eps}{\epsilon}

\newcommand{\G}{\Gamma}

\title{The Cantor Riemannium}

\subjclass[2010]{30F10, 30D99, 30E25. }
\keywords{}

\author[R. P\'{e}rez-Marco]{Ricardo P\'{e}rez-Marco}
\address{CNRS, IMJ-PRG, Sorbonne Universit\'e, Paris, France}
\email{ricardo.perez-marco@imj-prg.fr}

\begin{document}

\begin{abstract} The Riemann surface of a holomorphic germ is the space generated by its Weierstrass analytic continuation.
As envisioned a century ago by \'Emile Borel, we generalize this construction using Borel monogenic continuation which gives 
a Riemannium space. Riemannium spaces
are metric, path connected, Gromov length spaces, not necessarily $\sigma$-compact.
Contrary to Borel's expectations, we construct an Riemannium space which can be fully described: The Cantor Riemannium.

\end{abstract}

\maketitle


\begin{figure}[h]
\centering
\resizebox{5cm}{!}{\includegraphics{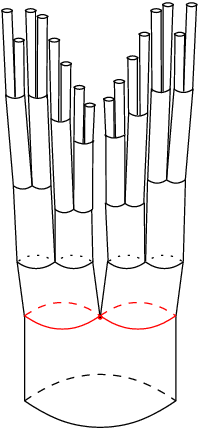}}    
\caption{The Krueger hand: Principal sheet of the Cantor Riemannium.} 
\end{figure}



\section{Introduction.}

The main construction in this article is originally inspired by an old forgotten  article of Giulio Vivanti (1888, \cite{Vi1}) 
about the cardinality  
of the values taken by a multivalued\footnote{Also named polydromic in the old literature.} analytic function at a given point, this is, 
all the values taken by all branches at this point. 

In this distant past, the notion of 
analytic continuation was not properly defined and Vivanti's article was dismissed and misunderstood. Also because 
Vivanti's pointed out that Poincar\'e didn't understand properly the ramification locus of a general function. 
If one adopts Weierstrass notion of analytic 
continuation, then the Cantor-Vivanti-Poincar\'e-Volterra 
Theorem\footnote{Nowadays improperly named Poincar\'e-Volterra Theorem.} states that the cardinal 
of branches at a point is always 
countable. The first published proof published by Vivanti (1888, \cite{Vi2}) was dismissed as erroneous. It 
suffers from the very same problems than 
Poincar\'e's earlier articles on general Riemann surfaces (for example \cite{Poi1}). Vivanti's first article 
\cite{Vi1} was motivated by these 
difficulties contained in Poincar\'e's article.

To understand the reason for the difficulties, we need to be aware of the historical context
of the original notion of Riemann surface during the XIX-th century. The birth of Riemann surfaces can 
be traced to the celebrated work of B. Riemann (1851, \cite{Rie1}, 1857, \cite{Rie2}) where he conceives them 
as the minimal space where the analytic continuation of a given holomorphic germ is well defined and single-valued. 
Before, people were talking about multiform functions, as one can read in Euler's works about the complex logarithmic function.
Riemann's revolutionary and non-intuitive idea was to pass from this ``multiformity'' of functions to a geometric 
multiformity of the domain of definition where the functions become single-valued. This breakthrough conception marked a point of transition 
to modern mathematics.

The original Riemann's Riemann surfaces had more 
structure that the modern notion proposed later by H. Weyl (1913, \cite{Wey}) and adopted everywhere today. 
They were domains over $\CC$, that  come equipped with a projection in $\CC$, i.e. they had a canonical coordinate chart. 
They also contain ``ramification points''.
Although the Riemann surface of the logarithm and its infinite ramification point was discussed by Riemann, the structure and proper 
definition of the ramification locus was only well understood for finite ramifications, 
which was important and sufficient to deal 
with the theory of algebraic functions and abelian functions (\cite{Rie2}), and 
isolated logarithmic ramifications, which appear in the fuchsian class of functions solutions of linear differential equations with 
polynomial coefficients studied by L. Fuchs (1866, \cite{Fu}). 
There was no clear understanding, nor a good definition, of more complicated ramification locus. 
This changed starting with an observation by A. Hurwitz in 1906 (\cite{Hur1}, \cite{Hur2}) about 
the singularities of inverse entire functions, when 
it became clear that more complicated singularities were natural.
In view of this historical context, the difficulty of understanding singularities is the central point  of 
Hilbert's 22nd problem (1900, \cite{Hil}). But few people today are aware of this 
because Weyl's modern notion of Riemann surface
excluded ramification points from its structure\footnote{This total absence of references 
to ramification points in Weyl's book is so blatant that one is led to suspect that 
the whole purpose of his new restricted definition was to remove these uncomfortable ramification points.}.

In 1931 S. Mazurkiewicz used a completion procedure to properly define singularities of functions (\cite{Maz}).
The original point of view was preserved  and has arrived to some modern schools  
in complex function theory (by the notable German and Russian schools of the 20th century), in particular 
in modern meromorphic function theory, where  
``concrete Riemann surfaces'' and ramifications appear as singularities of the functions studied. 
These functions are thought of as Riemann 
surfaces spread over the plane, 
the Riemann sphere, or in general onto another Riemann surface (see the classical books by Nevanlinna \cite{Nev1} 
and \cite{Nev2}, and the modern surveys by Eremenko \cite{Ere1}, \cite{Ere2}). In modern 
differential geometry presentations of Riemann Surfaces, ramification points are inexistent except 
for finite ramification points and their generalization to orbifolds in higher dimension, and related to conical singularities 
in Riemannian geometry. 

More recently, K. Biswas and the author have proposed another geometric approach by 
introducing the notion of log-Riemann surface that tries to recover faithfully the original geometric point of view 
of Riemann on Riemann surfaces (\cite{BPM1}). Roughly speaking, log-Riemann surfaces are Riemann surfaces endowed with a locally 
flat Euclidean metric pull-back by the Euclidean metric on the complex plane by the canonical projection. 
The ramification locus is obtained as the completion of log-Riemann surfaces for this 
log-Euclidean metric. This definition gives a larger and richer ramification locus than Markusewicz's one (we were unaware of Markusewicz's 
completion when proposing our theory of log-Riemann surfaces). This approach is parallel to the function 
theoretic point of view because to each holomorphic germ we can associate canonically a log-Riemann surface (see \cite{Le}). 
Conversely, sometimes, 
for example for simply connected log-Riemann surfaces and using the Uniformization Theorem, we can uniquely recover an entire function 
from the geometry of the log-Riemann surface.
But the points of view of both approaches are different, because the goals are different. 
Researchers in function theory study 
the complexity of singularities of functions. In the case of log-Riemann surface theory the goal 
is to study the geometric object by itself.
The ramification locus is part of the geometrical structure of the log-Riemann surface. 
This means that with the geometric point of view it makes
sense to consider functions extending to the ramification locus. This is very useful (see \cite{BPM4}) 
and fundamental for the development of the algebraic theory of base transcendental functions in transalgebraic 
curves (these are log-Riemann surfaces with a finite number of ramification points that can be infinite, contrary to classical 
algebraic curves). In section 5 of \cite{BPM5} a Dedekind-Weber transcendental theory (1882, \cite{DeWe}) is started. 
Classical Dedekind-Weber theory sets
the foundation of the algebraic theory of algebraic curves, is at the origin of Dedekind's definition of the notion of 
ideal that had a deep impact in Number Theory and the foundations of modern Algebraic Geometry. In log-Riemann surface theory we 
have geometric convergence notions different than those that are used in function theory and serve 
different purposes (an illustrative example can 
be found in \cite{BPM2} where a geometric proof of Euler's formula is given).
The point of view of log-Riemann surfaces is central to the present article, 
and the reader is supposed to have some familiarity with this theory.

Coming back to the historical context, years later of the inception of Riemann surfaces, there appeared new 
revolutionary ideas in complex function theoretic side.
\'Emile Borel discovered non-Weierstrassian analytic 
continuations and he introduced the notion of monogenic functions and Borel continuation. 
His insightful theory did mature from his Thesis (1894, \cite{Bo1}), was very much inspired by his research on resummation of series, 
and was finally presented in his monograph (1917, \cite{Bo2}). Borel's ideas met with strong oppostion. 
The introduction of Borel's book testifies of the opposition of the Weierstrassian school, 
in particular of G. Mittag-Leffler, to Borel's brilliant ideas \footnote{For more information, we refer to the historical article 
\cite{PMM} that has its origin on a manuscript note found by the author inside Mittag-Leffler's copy of Borel's book while he was visiting 
the Mittag-Leffler Institute in April 2023.}. 
Borel's monogenic class of functions is the precursor of quasi-analytic classes developed 
later by A. Denjoy (1921, \cite{De}) 
and T. Carleman (1926, \cite{Ca}). As was the case for Riemann, Borel's ideas were inspired by Cauchy's original point 
of view on holomorphic functions  as monogenic functions, i.e.
by $\CC$-differentiability, opposite to the Weierstrassian point of view of analytic functions (i.e. that are locally 
analytic series). It is instructive to observe how monogenicity survives at infinite ramification 
points (see section 5.4 of \cite{BPM5}).
Sometimes, these functions came from differential equations 
that yield non-Weierstrassian continuation properties. Borel insisted that the natural domain 
of definition for monogenic functions were more general domains that he named $C$ (for Cauchy) compared 
with the domains $W$ for Weierstrassian continuation that are always open sets. The typical 
domain $C$ that he describes in his book \cite{Bo2} can have empty interior. He deeply regretted that he 
believed that it was impossible to determine exactly the natural 
domain of extension $C$, contrary to the domains $W$ that were well understood. He thought to this as an 
inherent difficulty of his theory at the origin of the incomprehension of the Weierstrass school.
The monogenic classes of functions 
and domains $C$ considered by Borel are not the most general ones, as its residual measures are always 
atomic. It is relevant to 
remember what Borel wrote in the introduction of his book\footnote{``Obviously, all that is positive in the theory of Weierstrass analytic functions persists in 
the theory of non-analytic monogenic functions. But these positive aspects of the Weierstrass theory is not 
different except for minor details to those of Cauchy's theory, that used the analytic continuation by differential 
equations. Thus it is because of its negative aspects that Weierstrass theory owes its perfect logic 
that makes the admiration of their supporters. 
Replacing the domains W by the domains C we must apparently withdraw this perfection since apparently we will 
never be able to understand the exact limits where the new natural extensions are possible.''}\cite{Bo2}, 


\textit{``La th\'eorie des fonctions monog\`enes non analytiques laisse subsister, bien entendu, tout 
ce qu'il y a de positif 
dans la th\'eorie des fonctions analytiques de Weierstrass. Mais cette partie positive 
de la th\'eorie de Weierstrass ne 
se distingue que par des d\'etails de la th\'eorie de Cauchy, qui utilisait aussi le prolongement 
analytique par les \'equations 
diff\'erentielles; c'est \`a son c\^ot\'e n\'egatif que la th\'eorie de Weierstrass 
devait ce caract\`ere de perfection logique tant admir\'e
par ses adeptes; en substituant les domaines C aux domaines W, on doit renoncer \`a ce caract\`ere, 
car on ne pourra vraisemblablement 
jamais fixer les limites exactes au del\`a desquelles une extension nouvelle est impossible.(Borel, 1917)}


The main goal of this article is to present an explicit example where the Cauchy-Borel extension 
of a holomorphic germ can be fully understood, contrary to Borel's expectations expressed in the previous quote. 
It gives the first non-trivial example of 
Riemannium space, and it has many interesting properties. The function generating this Riemannium space 
has an uncountable number of branches obtained by Borel continuation that 
take an uncountable number of values at a point. This confirms and validates Vivanti's original intuition that 
turns out to be correct for Borel continuation instead of Weierstrass continuation. 

The example is build using the primitive of a function that is a Cauchy transform of a measure supported on a Cantor set. 
This function is monogenic in a maximal natural space 
that we determine explicitly.
As explained before, Weierstrass analytic continuation defines only a Riemann 
surface associated to a holomorphic germ, and even more, 
a concrete Riemann surface in the original sense 
of Riemann, or a log-Riemann surface as discussed before. 
The minimal space where the Borel monogenic extension of a holomorphic germ is single-valued defines 
a unique \textit{Riemannium space}. This natural space is not a manifold in general. It is a metric and 
path connected space but it is not in general $\sigma$-compact (as all Riemann surfaces are, by Rado's Theorem \cite{Ra}). 
It is ``concrete'' in the sense 
that it comes equipped, as do log-Riemann surfaces, with a natural projection on $\CC$ (it is a 
$\CC$-space in the sense of Bourbaki \cite{Bou}). 
The Riemannium space can contain a dense open set composed by an uncountable number of connected 
Riemann surfaces. The remaining part, the ramification locus,  has a rich structure but does not have a surface structure, and 
in general is more likely to have a fractal structure.

The goal of this article is to construct the explicit example of Riemannium space that is not a log-Riemann surface: 
The Cantor Riemannium.

\section{The Cantor tube-log Riemann surface.}

\subsection{Cauchy transform of a measure with finite total variation.}
In this first section $K\subset \CC$ is an arbitrary compact set. In the applications in 
subsequent sections $K$ will be a well defined Cantor set.
Let $\mu$ be a Borel measure on $\CC$ with total variation
$$
|\mu| (\CC) < +\infty 
$$
and compact support $\supp \mu = K\subset \CC$.
We consider the Cauchy transform\footnote{Another normalization in the harmonic analysis literature 
is $C_\mu(z)=\int_{\CC} \frac{d\mu(t)}{z-t}$ (see \cite{To} for the current theory).} of $\mu$ 
\begin{equation}\label{eq:CT}
f_\mu(z)=\frac{1}{2\pi i}\int_{\CC} \frac{d\mu(t)}{z-t} \ .
\end{equation}
Observe that if $\mu$ is an atomic measure with a finite number of atoms, 
then $f_\mu$ is a linear combination of simple polar parts,  
the poles being the location of the atoms, and the residues the mass at each point. 
If $\displaystyle{\mu =\sum_{n=1}^N a_n \delta_{z_n}}$, then
$$
f_{\mu} =\frac{1}{2\pi i}\sum_{n=1}^N \frac{a_n}{z-z_n}\ .
$$
\begin{proposition}
Let $K\subset \CC$ be the compact support of $\mu$. The Cauchy transform
$f_\mu$ is holomorphic on $\CC-K$.

Let $(\mu_n)_{n\geq 0}$ be a sequence of Borel measures  with uniformly bounded total variation, i.e. 
there exists $M>0$ such that for all $n\geq 0$,
 $$
|\mu_n| (\CC) \leq M <+\infty
$$
If $\mu_n\to \mu$ in the weak topology, and $\supp \mu_n \to \supp \mu =K$ in Hausdorff topology, then uniformly on compact sets of $\CC-K$ we have
$f_{\mu_n} \to f_\mu$.
\end{proposition}

\begin{proof}
Outside of $K$, since $\mu$ has finite total variation, the integral (\ref{eq:CT}) 
defining the Cauchy transform is absolutely convergent, and we can exchange derivation and integration signs, 
for $z\in \CC-K$
$$
\bar\partial_z f_\mu(z) =\frac{1}{2\pi i}\int_\CC \bar\partial_z \left (\frac{1}{z-t} \right ) \, d\mu(t) =0
$$
Alternatively, we can argue that atomic measures with a finite number of atoms is a weak-dense 
set in the space of measures with uniformly bounded total variation. 
Also we can approximate $\mu$ by such measures $(\mu_n)$ with support in $K$. The Cauchy transforms $f_{\mu_n}$ are meromorphic functions, holomorphic on $\CC-K$, 
and by Lebesgue dominated convergence Theorem they converge uniformly on compact sets of $\CC-K$ to $f_\mu$, hence $f_\mu$ is holomorphic on $\CC-K$.

For the second statement we observe that we have also $f_{\mu_n}\to f_\mu$ uniformly on compact sets of $\CC-K$ by Lebesgue dominated convergence Theorem.
\end{proof}

Thus the integral defines a locally holomorphic function at each point of $\CC-K$. 

\begin{proposition}[Monogenic Residue Formula]\label{prop:mon_residue}
Let $\Omega$ be a  Jordan domain such that  $\gamma=\partial \Omega \subset \CC-K$.
Then we have
$$
\int_\gamma  f_\mu (z) \, dz =  \mu (\Omega)
$$
\end{proposition}

\begin{proof}
The classical residue formula for the meromorphic functions $f_{\mu_n}$ approximating $f_\mu$ with atomic measures $(\mu_n)$ gives
$$
\int_\gamma  f_{\mu_n} (z) \, dz = \mu_n (\Omega)
$$
and we pass to the limit $n\to +\infty$.
\end{proof}

With some extra assumptions we can extend this result to the case when $\gamma \cap K \not= \emptyset$.

\begin{proposition} \label{prop:mon_residue2}
Let $\Omega$ be a  Jordan domain and  $\gamma=\partial \Omega$. Let $(\mu_n)$ be a sequence 
of measures converging to $\mu$ with uniformly bounded total variation.
If $V_\epsilon (\gamma)$ is the $\epsilon$-neighborhood of $\gamma$ we assume that
$$
\lim_{n\to +\infty} \mu_n(V_\epsilon (\gamma)) =0 \ .
$$
Then, the limit measure $\mu$ has no mass on $\gamma$, and we have
$$
\lim_{n\to +\infty} \int_\gamma  f_{\mu_n} (z) \, dz = \mu (\Omega)
$$
In particular, if $\gamma$ is rectifiable and $\mu(\gamma)=0$, and $(f_{\mu_n})$ converges uniformly to a continuous function $f_\mu$ on $\gamma$, then
$$
\int_\gamma  f_{\mu} (z) \, dz = \mu (\Omega)
$$
\end{proposition}

\begin{proof}
We can consider the restriction $\nu_{n, \eps}$ of the measure $\mu_n$ to $\CC-V_\epsilon (\gamma)$. Then the path $\g$ is now outside 
the support of the measure $\nu_{n, \eps}$ and by the previous Proposition we have
$$
\int_\gamma  f_{\nu_{n, \eps}} (z) \, dz = \nu_{n, \eps} (\Omega)
$$

Because of the uniform boundedness 
of the total variation, if we denote by $\mu'_n$ the restriction of the measure $\mu_n$  to $\CC-\gamma$, then 
we have that $f_{\nu_{n, \eps}}\to f_{\mu'_n}$ and $\nu_{n, \eps} (\Omega) \to \mu'_n(\Omega)$ when $\eps\to 0$ uniformly on $n$, then 
we get 
$$
\int_\gamma  f_{\mu'_{n}} (z) \, dz = \mu'_{n} (\Omega)
$$
Finally, since the limit measure $\mu$ has no mass on $\gamma$, we have that $f_{\mu'_{n}} \to f_\mu$ and $\mu'_n (\Omega) \to \mu(\Omega)$. 
Passing to the limit $n\to +\infty$, we have the result. 

\end{proof}

\subsection{The triadic Cantor set with its equilibrium measure.}

We study an explicit example of the above.
Let $C\subset [0,1]$ 
be the usual triadic Cantor set. It can be generated by the two affine maps $A_0(x)=3x$ and $A_1(x)=3x-2$ 
as the set of points 
$$
C=\bigcap_{n\geq 0} \bigcup_{\epsilon\in (\ZZ/2\ZZ)^n} \left (A_{\epsilon_1}\circ \ldots A_{\epsilon_n} \right )^{-1}([0,1])
$$
We can identify the Cantor set $C$ with the ring of $2$-adic integers $\ZZ_2$. The Haar measure $\mu_{\ZZ_2}$ 
gives the equilibrium probability measure $\mu$ on $C$. Observe that we can obtain $\mu$ as the weak limit 
of purely atomic probability measures $\mu_n$ that are the Haar measures of $\ZZ_2 /2^n \ZZ_2\approx \ZZ/2^n\ZZ$. Considering  
the standard representation of the 
triadic Cantor set, these are atomic measures composed by atoms of equal mass $2^{-n}$ at the $2^n$ 
end-points of the removed intervals 
from $[0,1]$ at depth $n\geq 1$. For instance, we have 
\begin{align*}
\mu_1 &= \frac12 \delta_0 + \frac12\delta_1 \\
\mu_2 &= \frac{1}{4}  \delta_0 + \frac{1}{4} \delta_{1/3} + \frac{1}{4} \delta_{2/3} + \frac{1}{4} \delta_1 \\
\mu_3 &= \frac{1}{8}  \delta_0 + \frac{1}{8} \delta_{1/9} + \frac{1}{8} \delta_{2/9} + 
\frac{1}{8} \delta_{1/3} +\frac{1}{8} \delta_{2/3} +\frac{1}{8} \delta_{7/9}+\frac{1}{8} \delta_{8/9}+\frac{1}{8} \delta_{1}\\
 &\vdots\\
\end{align*}
The triadic Cantor set $C$ just described has Hausdorff dimension $\dim_H C= \frac{\log 2}{\log3}$. The Hausdorff measure of this dimension 
coincides with $\mu$.

\subsection{Monodromy of the primitive of the Cauchy transform of the equilibrium measure.}

We consider the translated Cantor set $K=T_{-1/2}(C)$, where $T_a(z)=z+a$ is the translation, 
so that the origin $z=0$ is the center of symmetry. 
Let $\left (T_{-1/2}\right )_*\mu$ be the transported measure that we still denote $\mu$.

We consider the Cauchy transform
$$
f_\mu(z)=\frac{1}{2\pi i}\int_{\CC} \frac{d\mu(t)}{z-t}\ .
$$
By symmetry $f_\mu$ is odd and we have $f_{\mu}(0)=0$. We consider the primitive
$$
F_\mu(z)=\int_0^z f_\mu(t)\, dt
$$
which defines a holomorphic germ in a neighborhood of $0$ and is multivalued on $\CC-K$.  
Given a loop $\gamma \subset \CC-K$, with base point at $0$ and homotopy class
$[\gamma]\in \pi_1(\CC-K, 0)$, by analytic 
continuation we associate a value of the branch of $F_\mu$ along $\gamma$ that only depends on its homotopy class,
$$
\gamma \mapsto \Mon_\gamma(F_\mu)\ .
$$
The monodromy map to the additive group of complex numbers $\Mon: \pi_1(\CC-K, 0) \to (\CC, +)$ constructed in that way is a group morphism.

\begin{proposition}\label{prop:W-mono}
The monodromy of $F_\mu$ at $z=0$, that is the set of values that the multivalued function $F_\mu$ can take at $z=0$, is the additive group
of dyadic rationals,
$$
\Mon_{z=0} \, F_\mu =\bigoplus_{k\geq 0} 2^{-k} \ZZ =2^{-\infty} \ZZ
$$
\end{proposition}

\begin{proof}
Consider a Jordan loop $0\in \gamma \subset \CC-K$ containing the base point  $z=0$ and not homotopic to $0$. 
Then, taking $k\geq 1$ large enough, $\gamma$ encloses in its bounded 
component exactly $1\leq n(\gamma)\leq 2^k-1$ pieces of level $k$ of the Cantor set $K$. Then, 
using the Monogenic Residue Formula from Proposition \ref{prop:mon_residue} we have
$$
\int_\gamma  f_{\mu} (t) \, dt =  n(\gamma) 2^{-k}
$$
and this proves that the monodromy is in the dyadic rationals. On the other hand  enclosing 
a single fundamental piece at level $k$ and winding around 
the appropriate number of times we can achieve any monodromy of the form $ n2^{-k}$ for arbitary $n\in \ZZ$ given beforehand.
\end{proof}

\subsection{The Cantor tube-log Riemann surface.}
We refer to \cite{BPM1} and \cite{BPM2} for 
the definition of log-Riemann surfaces and to \cite{BPM3} for tube-log Riemann surfaces and a general description of tube-log Riemann 
surfaces associated to primitives of rational functions. 

The Cantor tube-log Riemann surface $\cS_\mu$ is the tube-log Riemann surface associated to $F_\mu$.
This means that the map $F_\mu$ is the uniformization of $\CC-K$ into a tube-log Riemann surface $\cS_\mu$ that we are now describing. 
According to \cite{BPM3}, to each rational function $R\in \CC(z)$, with set of poles and zeros $\cP_R$ (we assume that $0$ is not a pole), 
there corresponds a unique, up to normalization, pointed tube-log Riemann 
surface $(\cS_R, z_0)$ such that if 
$$
F(z)=\int_{0}^z R(t) \, dt
$$
the germ of $F$ at $0$ extends into the uniformization $F: \overline{\CC}- \cP_R \to \cS_R$  with $F(0)=z_0$.

The rational function $f_{\mu_n}$ has only simple poles of equal residues $2^{-n}$. The description of the monodromy of $F_{\mu_n}$ proves that 
the associated tube-log Riemann surface $\cS_{\mu_n}$ is obtained by pasting a sequence of copies of annuli $A_m$, $0< m\leq +\infty$, 
$$
A_m  = \{z\in \CC; 0< \Re z < m \}/(2\pi i)\ZZ
$$
We use $1+2+\ldots +2^n =2^{n+1}$ annuli glued together as a dyadic tree, by gluing their boundary by a translation map (the reader can find a similar construction in \cite{PM1})
and we obtain the surface in Figure 2. The first and the last annuli have infinite modulus and the others finite. All critical points of $F_{\mu_n}$ are real and there 
is exactly one critical point of $F_{\mu_n}$ inside each interval determined by consecutive poles.

\begin{figure}[h]
\centering
\resizebox{10cm}{!}{\includegraphics{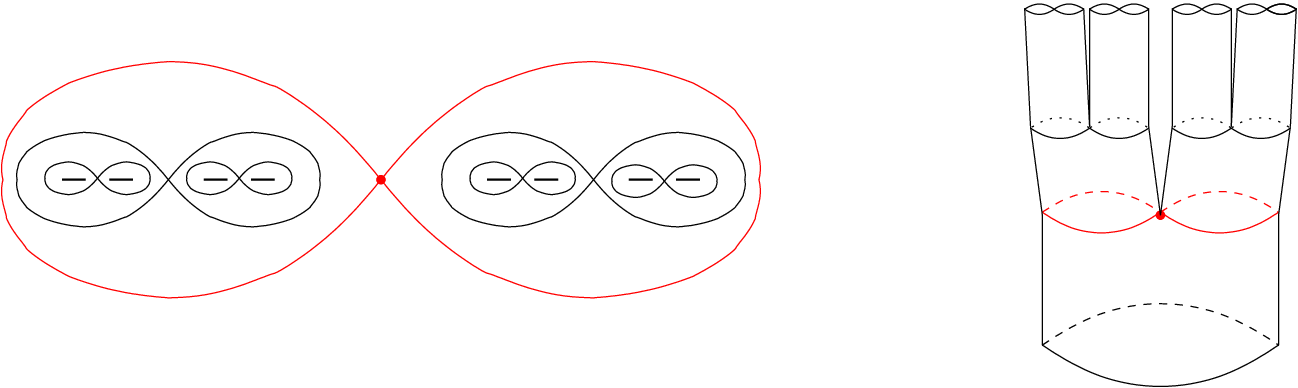}}    
\put (-100,45) {\scriptsize $\longrightarrow$}
\put (-103,54) { $F_{\mu_n}$}
\caption{Tube-log Riemann surface of $F_{\mu_n}$.} 
\end{figure}

For the primitive $F_\mu$, the associated tube-log-Riemann surface $\cS_{\mu}$ is constructed in a similar way but with an infinite number of annuli. There is 
exactly one critical 
point inside each removed interval generating the Cantor set.
The sequence of moduli along a branch that defines an end add up to infinity because of self-similarity. Consider the two annuli, symmetric with respect to $0$,
\begin{align*}
B_0 &= B(1/2,3/4)-\overline{B} (1/6, 1/4) \\
B_1 &= B(1/2,3/4)-\overline{B} (5/6, 1/4) \\
\end{align*}
Then given the point $z_{\boldsymbol{\epsilon}} \in K$, we have a nested of essential annulus defining fundamental neighborhoods of $z_{\boldsymbol{\epsilon}}$, 
$B_{\epsilon_0}, A_{\epsilon_1}(B_{\epsilon_0}), A_{\epsilon_2}\circ A_{\epsilon_1}(B_{\epsilon_0}), \ldots$ All these annuli have the same moduli and
$$
\sum_{n=0}^{+\infty} \mod A_{\epsilon_n}\circ \ldots \circ  A_{\epsilon_1}(B_{\epsilon_0}) =+\infty
$$
Using a standard criterium (see \cite{PM1} Lemma 2.17) we have that $\CC-K$ is in the $O_{AD}$ class of Riemann surfaces. This implies the following:
\begin{proposition}
Let $z\to z_{\boldsymbol{\epsilon}}\in K$ along a path $\eta \subset \CC-K$, then
$$
\lim_{z\to z_{\boldsymbol{\epsilon}}} \Im F_\mu (z) =+\infty
$$   
\end{proposition}
The tubular ``fingers'' have infinite height. This can also be established by a direct computation.

\medskip

\textbf{Green lines.}

\medskip

The lines corresponding to level lines of $\Re F_\mu $ are the Green lines orthogonal to the equipotentials\footnote{These are not the regular equipotentials of $K$.} 
that are the level lines of $\Im F_\mu$. We observe that we have a parametrization of Green lines by $\TT=\RR/\ZZ$ and we can normalize it so that 
the two Green lines hitting $0$ are $1/4$ and $-1/4$. We denote $\TT_0$ the abstract circle without the dyadic angles,
$$
\TT_0=\TT-\bigcup_{k\geq 0} 2^{-k}\ZZ \ .
$$
The inner points $K_i\subset K$ are those points in $K$ that are not end-points of the removed intervals.
To each inner point $x\in K_i$ of $K$ set there correspond an angle $\theta \in \TT_0$ which is not a dyadic rational such that the Green line of $\theta$ and $-\theta$ land
at $x$. We define the map $\sigma: \TT_0\to K_i$, $\theta\mapsto \sigma(\theta)=\sigma(-\theta)=x$.

We observe that the prime-ends at infinity of $\cS_\mu$ correspond to the point $\infty$ and the points in the Cantor set, and the 
inner points in the Cantor set $K$ correspond to two 
prime-ends for the lines $\Re F_\mu =\pm \theta$. So there is a natural quotient of the prime-end compactification that gives the Riemann sphere.

\subsection{Julia Cantor sets.}

For those readers familiar with holomorphic dynamics, we point out that we could construct examples of Cantor sets of the above type by 
taking the Julia set of a quadratic polynomial $P_c(z)=z^2+c$ with $c\notin M$ outside the Mandelbrot set, this means that the 
iterates of the critical point $z=0$ escapes to infinity, $P_c^n(0)\to \infty$. In that case the Julia set is the set of points with 
bounded (positive) orbit, 
$$
J_c=\{ z\in \CC; (P_c^n(z))_{n\geq 0} \text{ bounded} \} 
$$
The Julia set $J_c$ obviously contains all periodic orbits and it is the closure of the set of periodic orbits. 
It is a Cantor set and we take for $\mu$ the measure of maximal 
entropy that is equivalent to the limit of equidistributed atomic measures on periodic orbits. Then the B\"ottcher coordinate of the basin of attraction 
at infinite is defined in the tube-log Riemann surface described before. The reader can consult \cite{PM1} and \cite{Em} for how to determine 
combinatorially the tube-log Riemann surface.

\section{The Cauchy transform of a singular measures and its Borel monodromy.}

We define now a new measure $\nu$ on $K$ that is singular with respect to the equilibrium measure $\mu$. 
By construction $\nu$ is an atomic measure, but we could easily modify the construction to get a non-atomic measure.
Let $(\lambda_n)_{n\geq 0}$ be a sequence of positive reals $\lambda_n>0$ decreasing fast to $0$, $\lambda_n\to 0$. 
We assume that 
\begin{equation} \label{eq:somme}
\sum_{k=0}^{+\infty} 2^{k} \lambda_k =1 \ .
\end{equation}
We consider the probability measure
$$
\nu =\sum_{k=0}^{+\infty} \lambda_k \sum_{l=1}^{2^k} \delta_{x_{k,l}} 
$$
where $(x_{k,l})_l$ are the end-points of order $k\geq 0$ defining the Cantor set $K$.

\begin{proposition}
We assume that the sequence $(\lambda_k)_{k\geq 0}$ decreases fast enough so that for some constant $\kappa >3$ 
$$
\sum_{k=0}^{+\infty} \lambda_k (2\kappa)^{k} <+\infty \ .
$$
For a dense and $F_\sigma$ set of good points $G_0\subset K$ of full $\mu$-measure, we have that for $x_0\in G_0$ the 
primitive of the Cauchy transform 
$F_{\nu} $ is a continuous function on the vertical line $x_0+i\RR$. In particular, there is a finite limit
$\lim_{y\to 0} F_{\nu} (x_0+iy) =F_{\nu}(x_0)$.
\end{proposition}

\begin{proof}
Consider the set $G_0\subset K$ of points badly approximated by end-points $(x_{k,l})$, 
$$
G_0 = K-\bigcap_{k_0\geq 0} \bigcup_{k\geq k_0} \bigcup_{l=1}^{2^k} B(x_{k,l},\kappa^{-{k}})
$$
Then $G_0$ is a dense and  $F_\sigma$ set and of total $\mu$-measure.

For any $x_0\in K_0$, and $z=x_0+iy$, $y \in \RR$, we have for some $k_0\geq 1$, for $k\geq k_0$, $1\leq l\leq 2^k$,
$$
\left |\frac{\lambda_k}{z-x_{k,l}} \right | \leq  \left |\frac{\lambda_k}{x_0-x_{k,l}} \right | \leq \lambda_k \kappa^k
$$
and 
$$
\left |\sum_{l=1}^{2^k} \frac{\lambda_k}{z-x_{k,l}} \right | \leq \lambda_k (2\kappa)^k
$$
thus condition (\ref{eq:somme}) proves that the series
$$
F_{\nu} (z)=\frac{1}{2\pi i}\sum_{k=0}^{+\infty} \sum_{l=1}^{2^n} \frac{\lambda_k}{z-x_{k,l}} 
$$
is normally convergent and bounded, and defines a continuous function on $x_0+i\RR$.
\end{proof}

\begin{proposition}
The primitive of the Cauchy transform  $F_{\nu}$ is well defined, continuous 
and bounded on the full Lebesgue set $(\CC-K)\cup G_0$ and holomorphic in $\CC-K$.
\end{proposition}

\begin{proof}
We only need to  check continuity at points $x_0\in G_0$ and this follows observing that the series in the precedent proof converges normally in $G_0$.
\end{proof}

So, we can extend continuously the function $F_{\nu}$ on vertical lines $x_0+i\RR$ with $x_0 \in G_0$. This holds for any path 
crossing the real axis at $x_0 \in G_0$ with a non-zero angle. This defines the Borel extension of $F_{\nu}$ on this path 
(see section \ref{sec:Borel_ext}). We can define the larger set $G_1\subset K$ where this extension holds.

\begin{definition}
 Let $G_1\subset K$ be such that for $x_1\in G_1$ we have that the restriction of $F_{\nu}$ to $(x_1+i\RR)-\{x_1\}$ has a continuous extension 
 to the full vertical line $x_1+i\RR$. We have $G_0\subset G_1\subset K$ and the 
 set $P_{\nu}=K-G_1$ is defined to be the monogenic polar set (in the sense of support of the ``poles'' of the monogenic 
 extension, where it becomes infinite).
\end{definition}

\begin{theorem}\label{th:Borel_monodromy}
The values of different branches of $F_{\nu}$ at $z=0$ by Borel continuation is the Borel monodromy
$$
B_{\nu}=\BMon_{z=0} F_{\nu}\subset \RR
$$ 
that is an uncountable additive sub-group. In particular, it contains the uncountable subset
$$
\{\nu([0, x_1]); x_1\in G_1 \} \subset B_{\nu}
$$
The subgroup $M_{\nu}\subset \RR$ also contains the countable sub-group of the regular Weierstrass monodromy,
$$
W_{\nu} = \bigoplus_{n\geq 1} 2^{-n}\lambda_n \ZZ   \subset B_{\nu} \subset \RR \ .
$$
\end{theorem}

\begin{proof}
We can cross the Cantor set $K$ at the uncountable set of points in $G_1$, hence we can split the total measure in two parts of mass $\nu([-1/2,x_1])$ 
and $\nu([x_1,1/2])$ by a simple Jordan loop from $0$ crossing $K$ at $x_1\in G_1$. If the Jordan curve crosses the real line only at one point $x_1\in G_1$, and 
the Jordan curve is positively oriented, then 
the  Monogenic Residue Formula in Proposition \ref{prop:mon_residue} proves that we have the monodromy value $\nu([0, x_1])$. The computation of the 
regular Weierstrass monodromy follows from the same argument as in Proposition \ref{prop:W-mono} observing that the mass of each piece at level $n\geq 1$ building-up 
the Cantor set is $2^{-n}\lambda_n$.
\end{proof}

\section{The restricted Cantor Riemannium.}

\subsection{The Krueger hand, ramification locus and monogenic topology.}

As for the Riemann surface associated to the Weierstrass continuation, we want to understand the minimal space where the monogenic function $F_{\nu}$
is single-valued (or monodromic in ancient terminology). This is the Riemannium space associated to  $F_{\nu}$.

In the rest of this paragraph we give an informal geometrical description of this Riemannium 
space associated to $F_{\nu}$  (that aims to be useful to the intuition of the reader, but can be skipped). 
The Riemann surface $\cS_{\nu}$ associated to the Weierstrass continuation of $F_\nu$ is similar to the tube-log Riemann surface $\cS_\nu$ (see Figure 2). 
But this time it has an uncountable number of ``finger tubes'' with finite length. In these ``finite length fingers''  the imaginary part 
of $F_{\nu}$ is bounded from above. 
The tips of the fingers at some height $\leq H$ form a closed set that corresponds to a Cantor subset in $G_1$. The  external rays 
of bounded height (``finite length fingers'') form a smooth fan in the sense defined in \cite{Cha}. 
When the tips of the fingers are dense (this requires some homogeneous  
property of the measure $\nu$) then we have what is called a Lelek fan which has some unique topological characterization and was defined in \cite{Le}.  
We call the described geometric object  a ``Krueger hand'' (see Figure 1). The Krueger hand plays the role of the $0$-sheet of the Riemannium 
space associated to $F_{\nu}$ (as the $0$-sheet of the log-Riemann surface of the logarithm). 
The tips of the finger tubes of finite length do correspond to the subset with no isolated points $G_1\subset K$, 
and can be encoded with external angles $\theta$ such that $\sigma(\theta)=x_1\in G_1$ 
(the map $\sigma$ map angles that are not dyadic rationals 
to the landing point in $K_i\subset K$ as in section 2.4). 
For these points $x_1\in G_1$, when we continue upward along the vertical 
$\sigma(\theta)+i\RR_-$ and we cross the tip finger corresponding to $x_1$, 
we appear in the corresponding 
upper vertical line $\sigma(-\theta)+i\RR_+$ (note that the 
line orientation is reversed when we cross). 
But then the monodromy changes by addition of the non-dyadic rational $\theta \in \TT_0$. 
Thus, we do appear in a new copy. 

The tube-log Riemann surface $\cS_{\nu}$ comes equipped 
with a log-euclidean riemannian metric $|dz|$ coming from the canonical coordinate $z$ 
(defined up to a translation, hence the metric $|dz|$ is well defined), 
and we can define its ramification locus as 
its Cauchy completion as in \cite{BPM1}, \cite{BPM3} and \cite{L}.

\begin{definition}[Ramification locus]
The ramification locus $\cR_{\nu}$ of $\cS_{\nu}$ is the completion of $\cS_{\nu}$ for its log-Euclidean metric.
The completion $\cS_{\nu}^*=\cS_{\nu}\cup \cR_{\nu}$ is a path connected length space extending the monogenic metric. 
\end{definition}

Now, taking a pull-back with $F_{\nu}$, we can consider the monogenic metric on $\CC-K$.

\begin{definition}[Monogenic metric]
The monogenic metric on $\CC-K$ is defined as the pull-back of the log-Euclidean metric 
on $\cS_{\nu}$ by the map $F_{\nu}: \CC-K\to \cS_{\nu}$.
\end{definition}

Recall that $P_{\nu}=K-G_1$ is the polar set where $F_{\nu}$ becomes infinite. 
The map $F_{\nu}$ extends continuously to $F_{\nu}: \CC-P_{\nu}\to \cS_{\nu}^*$. We can 
embed $\cS_{\nu}^*$ into $\CC$ by extending 
the inverse $F_{\nu}^{-1}: \cS_{\nu} \to \CC-K$ into $F_{\nu}^{-1}: \cS_{\nu}^* \to \CC-P_{\nu}$.   
Since the log-Euclidean metric extends to the completion $\cS_{\nu}^*$, we can pull-back the metric 
space structure by $F_{\nu}$ and define in $\CC-P_{\nu}$ the monogenic metric. The associated topology on this set is 
the monogenic topology.

\begin{proposition}
The monogenic topology of $\cS_{\nu}^*$ 
defines a topology on $\CC-P_{\nu}$ that is weaker that the trace topology or the plane. 
The space $\CC-P_{\nu}$ is exhausted by the 
the monogenic closed metric balls $(\bar B(z_0,n))_{n\geq 1}$, so it is an $F_\sigma$ set 
for the monogenic topology. The closed monogenic 
balls are not closed for the standard topology. 
\end{proposition}

Borel attempted to define some elements of the monogenic topology (see in \cite{Bo2} the description 
of $C$ domains in Chapter V, pages 125-134).

\subsection{The Restricted Riemannium space associated to $F_{\nu}$.}
By pasting copies of the Krueger hand,  can define now a Restricted Cantor Riemannium $C^r_{\nu}$ 
(super-index r for ``restricted'') that is the space where the Borel extension of the 
monogenic function $F_{\nu}$ only along paths crossing $G_1$ in a finite set. 
The natural extension space is larger when we consider paths crossing $G_1$ in a countable set, 
and we will consider it 
in the next section. The construction in this section is very concrete and worth doing it first. It is the same 
idea that when we build the log-Riemann surface of the logarithm using a set of copies of $\CC$ indexed by the 
integers $\ZZ$ and paste them accordingly. Here the fundamental sheet is a Krueger hand, and we have many more ways of 
gluing the sheets together that is described in the next definition.

\begin{definition}[Restricted Cantor Riemannium]\label{def:Cantor_Riemannnium}
We construct the Restricted Cantor Riemannium  $C^r_{\nu}$ as follows:
\begin{itemize}
 \item We consider copies $(\cS^*_{\nu} (\boldsymbol{\theta}))_{\boldsymbol{\theta}}$ 
 of the completed Krueger hand $\cS^*_{\nu}$ indexed by finite sequences 
 $\boldsymbol{\theta} =(\theta_n)_{1\leq n\leq n(\boldsymbol{\theta})}$ of angles $\theta_n \in \sigma^{-1}(G_1)$ (this is 
 an uncountable number of indexes, hence of copies, and we use the axiom of choice). 
 We include a copy  $\cS^*_{\nu} (\emptyset)$ corresponding to the empty sequence, which is named  the principal sheet.
Let $j_{\boldsymbol{\theta}} :\cS^*_{\nu} \to \cS^*_{\nu}(\boldsymbol{\theta})$ be the mapping identifying these copies.
\item We glue to the  copy $\cS^*_{\nu} (\boldsymbol{\theta})$ all copies 
$\cS^*_{\nu} (\boldsymbol{\theta'})$, such that $ n(\boldsymbol{\theta'})=
n(\boldsymbol{\theta})+1$ at a single point by identifying the point 
$F_{\nu}(\sigma(\theta'_{n(\boldsymbol{\theta'})})) \in \cS_{\nu} (\boldsymbol{\theta})$ to 
the point $F_{\nu}(\sigma(- \theta'_{n(\boldsymbol{\theta'})})) \in \cS^*_{\nu} (\boldsymbol{\theta'})$
when $\sigma(\theta'_{n(\boldsymbol{\theta'})})=x_0\in G_1$.
That is, for $x_0 \in G_1$, $x_0 =\sigma(\theta'_{n(\boldsymbol{\theta'})}) =\sigma(- \theta'_{n(\boldsymbol{\theta'})})$
$$
j_{\boldsymbol{\theta}}(F_{\nu}(x_0)) = j_{\boldsymbol{\theta'}}(F_{\nu}(x_0))
$$
The space obtained $\hat C^r_{\nu}$ is path connected and has a well-defined continuous projection map 
$$
\hat \pi : \hat C_{\nu} \to \cS^*_{\nu} 
$$
such that for $z\in \cS^*_{\nu} (\boldsymbol{\theta})$, $\hat \pi (z) = \pi (j_{\boldsymbol{\theta}}^{-1}(z))$.
\item Any glued copy $\cS^* (\boldsymbol{\theta})$  in $\hat C^r_{\nu}$ can be connected to 
the principal sheet $\cS^*_{\nu} (\emptyset)$
by a path $\gamma_{\boldsymbol{\theta}}$ starting at $j_\emptyset (0) \in \cS^*_{\nu} (\emptyset)$, 
ending at $j_{\boldsymbol{\theta}}(0) \in \cS^*_{\nu} (\boldsymbol{\theta})$, 
going through 
$\cS^*_{\nu}(\boldsymbol{\theta}_1)= \cS^*_{\nu}(\boldsymbol{\theta}_2),\ldots , 
\cS^*_{\nu}(\boldsymbol{\theta}_{n(\boldsymbol{\theta})})$,
where $\boldsymbol{\theta}_j =(\theta_1,\ldots , \theta_j)$, and crossing a finite number of 
sewing points transversally to what corresponds to the real axes in each copy.

We denote by $M(\boldsymbol{\theta}) \in \CC$ the Borel monodromy of $F_{\nu}$ along $\hat \pi (\gamma_{\boldsymbol{\theta}})$ 
that is well defined by virtue of Theorem \ref{th:Borel_monodromy}.

\item We identify the copies $\cS^* (\boldsymbol{\theta})$ and $\cS^* (\boldsymbol{\theta'})$ 
if and only if $M(\boldsymbol{\theta})$ and $M(\boldsymbol{\theta'})$ are 
equal modulo the dyadic rationals, 
\begin{equation}\label{eq:modulo_mon}
M(\boldsymbol{\theta})\equiv M(\boldsymbol{\theta'}) \ \ \left [\text{mod} \  \ \bigoplus_{k\geq 0} 2^{-k} \lambda_k\ZZ \right ] 
\end{equation}
This defines an equivalence relation in $\hat C^r_{\nu}$ and the topological quotient $C^r_{\nu}=\hat C^r_{\nu}/\sim$ is the Restricted Cantor Riemannium.
\end{itemize}
\end{definition}

Observe that the projection mapping is compatible with the equivalence relation 
(equivalent classes are contained in the fibers of $\hat \pi$) and defines
a projection mapping (still denoted by $\hat \pi$),
$$
\hat \pi : C^r_{\nu} \to \cS^*_{\nu}\approx \CC-P_{\nu} 
$$

\begin{proposition}
The Restricted Cantor Riemannium $C^r_{\nu}$ is a Hausdorff topological space that is path connected and not $\sigma$-compact. It is a metric length space. 
The Restricted Cantor 
Riemannium has an open dense part composed by an uncountable number of disjoint log-Riemann surfaces for the projection $\hat \pi$,
$$
C^0_{\nu} =\hat \pi^{-1} (\cS_{\nu}) \subset C^r_{\nu}
$$
We call $C^0_{\nu}$ the regular part of the Riemannium.
The projection mapping $\hat \pi :  C^r_{\nu} \to \cS^*_{\nu}$
is a local holomorphic diffeomorphism and, for the log-Euclidean metric, a local isometry on 
the regular part, and a contraction elsewhere.
\end{proposition}

\begin{proof}
The space is not $\sigma$-compact since we attach to $\cS_{\nu}(0)$ an uncountable number of copies.
The log-euclidian metric of  $\cS_{\nu}$ defines the length space metric in the copies that build $C_{\nu}$. 
The connections between copies are through a single point so the length metric extends along the crossing paths and in the resulting length space 
the copies are embedded isometrically. The other properties are clear.
\end{proof}

Observe that $\hat \pi : C^r_{\nu} \to \cS^*_{\nu}$ is not a classical covering of topological spaces 
since for any point $\tilde x_1\in \hat \pi^{-1}(x_1)$ with  $x_1\in G_1$ there is no neighborhood of 
$\tilde x_1$ homeomorphic to a neighborhood of $x_1$. It is even worse: we have infinite local degree. 

We recall that the Uniformization Theorem holds in the category of tube-log-Riemann surfaces. Let $\tilde \cS_{\nu}$ be the universal covering 
of the tube-log-Riemann surface $\cS_{\nu}$ which means that  $\tilde \cS_{\nu}$ is a simply connected log-Riemann surface and we have a 
covering $p_0: \tilde \cS_{\nu} \to \cS_{\nu}$ compatible with the projection mappings $\pi_0: \tilde \cS_{\nu} \to \CC$, which means 
that locally $p_0 \circ \pi_0^{-1}$ is a translation, and $p_0$ is a local isometry for their log-Euclidean metrics. 

Another way to understand $\tilde \cS_{\nu}$ is to ``unfold'' the cylinders that build $\cS_{\nu}$. 
We get in this way a log-Riemann surface that can be directly constructed 
by means of straight cuts and pastings. In Figure 3 we show a fundamental domain with the 
obvious identifications (indicated by dotted lines) to get $\tilde \cS_{\nu}$ 
(for $\tilde \cS_{\nu}$ is similar, with widths corresponding to the mass $\nu$). 
The full $0$-sheet with cuts is obtaining by extending this fundamental domain by $\ZZ$-periodicity. This $0$-sheet corresponds to the ``unfolding'' of the root-cylinder. 
Note that the set of cuts in this $0$-sheet for $\tilde \cS_{\nu}$ do not have a discrete set 
of end-points (for $\tilde \cS_{\nu}$ this set is discrete in the $0$-sheet) .

\begin{figure}[ht]
\centering
\resizebox{4cm}{!}{\includegraphics{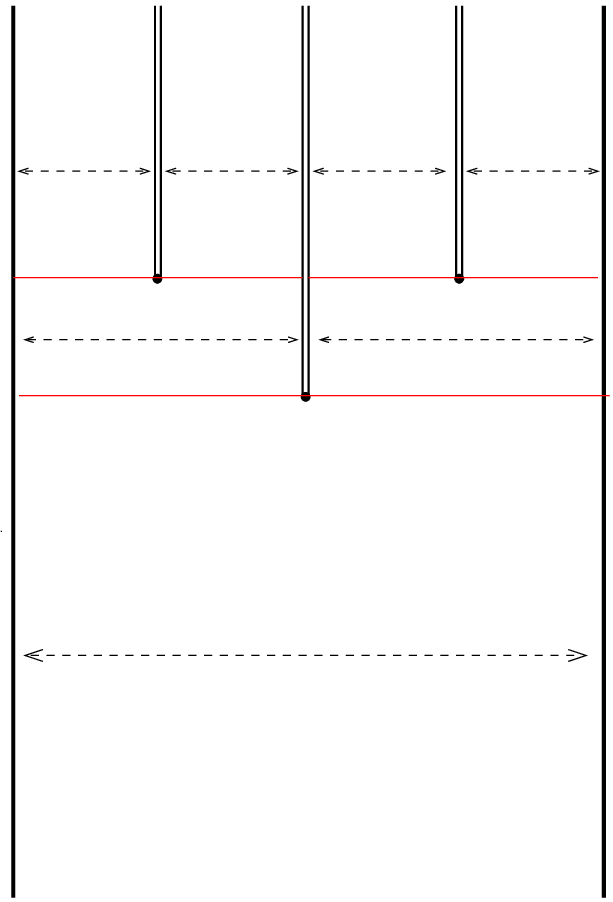}}    
\put (-57,37) {\scriptsize $1$}
\put (-87,98) {\tiny $\frac{1}{2}$}
\put (-31,98) {\tiny $\frac{1}{2}$}
\put (-102,127) {\tiny $\frac{1}{4}$}
\put (-74,127) {\tiny $\frac{1}{4}$}
\put (-46,127) {\tiny $\frac{1}{4}$}
\put (-18,127) {\tiny $\frac{1}{4}$}
\caption{Fundamental domain for $\tilde \cS_{\nu}$.} 
\end{figure}
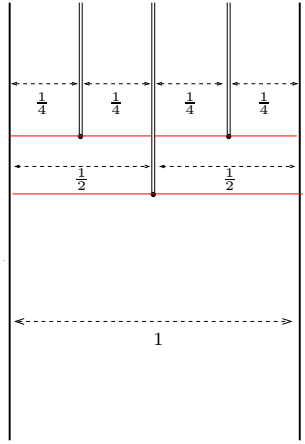

We can consider the  Cauchy completion of $\tilde \cS_{\nu}$ for the log-Euclidean metric and 
get $\tilde \cS^*_{\nu}$. Then 
we can start  with $\tilde \cS^*_{\nu}$ and carry out the same Cantor Riemannium construction as before. After the first two steps we 
get the space $\check C^r_{\nu}$  that is a simply connected topological space (all loops are contractible). 
The pasting construction is 
compatible with the projection $\pi_0 :\tilde \cS^*_{\nu} \to \CC$ defined on 
the copies, and this defines a projection $\check \pi : \check C^r_{\nu} \to \CC$.
The regular part of $\check C^0_{\nu}$ is defined as before.
Despite having a projection 
 $\check C^0_{\nu}\to  \CC-P_{\nu}$, this ``cover'' cannot be identified 
with the universal cover of $\CC-P_{\nu}$ because the lift property for paths fails. 
We can only lift paths in $\CC-P_{\nu}$ such that when we remove its 
end-point they do have a discrete intersection with $G_1$ transversal to $\RR$. If we consider loops in $\CC-P_{\nu}$ with base point 
at a regular point $z_0\in \CC-P_{\nu}$, for example $z_0=0$, with a finite intersection with $G_1$ and transversal with $\RR$ we have a subgroup 
$\pi^r_1(\CC-P_{\nu}, z_0)$ of $\pi_1(\CC-P_{\nu}, z_0)$. The group of ``deck transformations'' of the map 
$\check C_{\nu}\to  \tilde \cS^*_{\nu} \to \cS^*_{\nu}\approx \CC-P_{\nu} $ 
(it is not a classical cover as observed before)
is the quotient group  $\pi_1(\CC-P_{\nu}, z_0)/\pi^r_1(\CC-P_{\nu}, z_0)$.

Then we can consider the quotient of $\check C_{\nu}$ as in the last point of the construction in Definition \ref{def:Cantor_Riemannnium} 
and get 
$$
\tilde C^r_{\nu} = \check C^r_{\nu} /\sim
$$
where this time the equivalence relation $\cS^* (\boldsymbol{\theta})\sim \cS^* (\boldsymbol{\theta'})$ holds if and only if
$M(\boldsymbol{\theta})= M(\boldsymbol{\theta'})$.

We get in this way a restricted Riemannium space $\tilde C_{\nu}$  with a projection mapping
$$
\tilde \pi : \tilde C^r_{\nu} \to \CC
$$
and we have a fiber-compatible mapping
$$
\tilde p : \tilde C^r_{\nu} \to C_{\nu}
$$
that is a sort of ``universal cover Riemannium''. 
We have the correspondence of the regular parts
$$
\tilde p ( \tilde C^0_{\nu})= C^0_{\nu}
$$
and $\tilde p$ is a local isometry on the regular part for the log-Euclidean metrics. We have the commutative diagram

$$
\begin{array}{ccccc}
\check C^r_{\nu} & \longrightarrow{} & \tilde C^r_{\nu} & \longrightarrow{} & \tilde \cS^*_{\nu} \\
\downarrow  & & \downarrow & & \downarrow \\
\hat C^r_{\nu} & \longrightarrow{} & C^r_{\nu} & \longrightarrow{} &\cS^*_{\nu} \\
\end{array}
$$

We would like to carry out a similar construction as the classical algebraic topology construction  of the universal cover but we have several 
difficulties. The first one is related to the non-local simple connectivity of $\CC-P_{\nu}$ at 
points of $G_1$. We can construct Hawaian earrings at each point $x_1\in G_1$ by taking a bouquet of circles with base point at $x_1$ 
and crossing transversally $\RR$ at another point of $G_1$. But more importantly, we don't have the path lift property: There are paths that cannot 
be lifted to the Cantor Riemannium. For example, take a path $\gamma$ that wiggles between the upper and lower half planes determined by $\RR$ and converging to a point 
$x_\infty \in G_1$ and after reaching $x_\infty$ goes into the upper half plane (see Figure 4).

\bigskip

\begin{figure}[ht]
\centering
\resizebox{6cm}{!}{\includegraphics{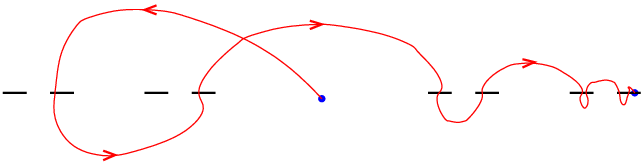}}    
 \put (-85,10) {\scriptsize $z_0$}
  \put (-2,12) {\scriptsize $x_\infty$}
\caption{Wiggling path.} 
\end{figure}

\bigskip

There is no lift of $\gamma$ in $\check C_{\nu}$ because there is no sheet where the end-point of the lift belongs to.

The reason for this phenomenon is related to the fact that the Riemanniums $\hat C^r_{\nu}$ and $\tilde C^r_{\nu}$ are not complete for the 
log-euclidean metric, despite that the sheets $ \tilde \cS^*_{\nu}$ and $\cS^*_{\nu}$ are complete.

\begin{proposition}
The restricted Riemanniums $\hat C^r_{\nu}$ and $\tilde C^r_{\nu}$ are not complete.
\end{proposition}

\begin{proof}
We can take a rectifiable path $\gamma$ of finite length for the log-Euclidean metric as before. 
And a monotone Cauchy sequence $(z_n)\subset$ converging in $\CC-P_{\nu}$ to $x_1$ 
for the log-Euclidean metric in $\CC-P_{\nu}$. There is a unique lift of 
$\gamma-\{x_\infty\}$ in $\hat C_{\nu}$ (or $\tilde C_{\nu}$) but it has no limit point because of the 
discrete topology on the fiber of $x_\infty$.
\end{proof}

The proof gives a hint on how to solve this problem: We must enrich the topology of  $\hat C^r_{\nu}$ and introduce a 
natural non-discrete topology on the fibers of $\hat C^r_{\nu} \to \CC-G_1$.

To each point in a fiber $z\in \hat \pi^{-1} (x_1)$ with $x_1 \in G_1$ there corresponds a 
sequence $\boldsymbol{\theta}(z)$ in the construction. We define a distance on the 
fiber by
$$
d_0(z_1, z_2)=|M(\boldsymbol{\theta}(z_1))- M(\boldsymbol{\theta}(z_2)|
$$
This defines a metric on each fiber that makes it complete. And the fiber is no longer discrete as for a classical covering.

Now, let $x_\infty\in G_1$ and a path $\gamma$ as before, starting at the regular base point $z_0$ and ending at $x_\infty$ avoiding $P_{\nu}$, thus $\gamma \subset \CC- P_{\nu}$, 
and crossing transversally the real axes except at the end-point. 
We can homotop $\gamma$ in $\CC-K$ into $\gamma'$ so that $\gamma\cap G_1= \gamma' \cap \RR$, by just moving the connected components of $\gamma-\G_1$ 
that intersect the real line into the upper or lower half plane. Now, for the sequence of intersection points  $(x_n)\in \gamma '\cap G_1$, $x_n\to x_\infty$ we 
know  in which sheet we are by looking at 
the value of the corresponding $M(\boldsymbol{\theta}_n)$. The limit of the sequence $(M(\boldsymbol{\theta}_n))_{n\geq 1}$ indicated in which sheet is the limit in 
the fiber $\pi^{-1}(x_\infty)$, and we have
$$
\lim_{n\to +\infty} M(\boldsymbol{\theta}_n) =M(\theta_\infty) 
$$
Then the path passes through the sheet determined by $M(\theta_\infty)$ and connects to the one corresponding 
to $M(- \theta_\infty)$. This suggests that we have to do a 
further quotient of $\check C^r_{\nu}$ on the fibers $\check \pi^{-1}(x_\infty)$. This will give the unrestricted 
Cantor Riemannium $\check C_{\nu}$.

But we would like to have this extension constructed not only for paths with a countable transversal intersection with $\RR$ but those with a non-discrete intersection with $G_1$.
The same procedure works, and for this purpose we do the general construction of the next section.

\section{The Cantor Riemannium.}

\subsection{Monogenic algebraic topology.}

We carry out a similar construction\footnote{I am indebted to K. Biswas to point out  this approach.} 
of the universal covering in algebraic topology but for the non-locally simply connected space $X=\CC-P$.
We refer to \cite{Ha}, \cite{Go}, and in particular to \cite{Bou} for minimal hypothesis for classical results for relation of 
fundamental groups and coverings. 
The natural topology on the fundamental group, is the \textit{admissible topology} defined by N. Bourbaki\footnote{Special thanks to Andr\'e Gramain.} 
in  \cite{Bou}, Chap. III, 5.4, p.315. In our non locally simply connected situation, the admissible topology is not 
discrete. Also we don't have coverings with discrete fibers, but a general projection with non-discrete fibers. The key property 
to carry out a similar construction of the universal covering is to use projections with the path-lifting property. After figuring this out 
we learned about the excellent exposition by E.H. Spanier in his classical book \cite{Sp}, in particular chapter 2, where all the strength of 
the path-lifting property is used for Hurewicz fibrations. But the projections we have are not a priori Hurewicz fibrations. 
The situation we have is closer to Serre fibrations in one topological dimension where we have the lifting of path homotopies. 

\medskip

We remind that the space $X=\CC-P$ (we simplify the notation in this section by removing the reference to $\nu$ that is fixed once for all) 
can be endowed with the regular planar topology or with the monogenic topology, that is more natural for our purposes.
For both topologies, the space $X$ is path-connected, hence connected, separable, and locally path-connected, but 
not semilocally simply connected, not even locally simply connected. More precisely, each point $x_1\in G_1$ is the base point for a uncountable Hawaian earring as is easily seen 
by taking a family of circles $(C_{x_1,x_2})_{x_2\in G_1}$ cutting perpendicularly the real axes at $x_1, x_2 \in G_1$ (See Figure 5).

\begin{figure}[ht]
\centering
\resizebox{6cm}{!}{\includegraphics{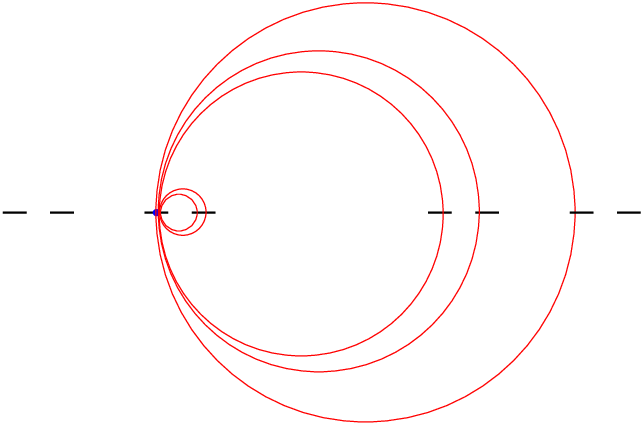}}    
 \put (-137,48) {\scriptsize $x_1$}
  \put (-40,48) {\scriptsize $x_2$}
\caption{Hawaian earring based on $x_1$.} 
\end{figure}

A path $\gamma$ in $\CC-P$ is essential  or we say that $\gamma$ has an essential intersection at $x_1\in \gamma \cap G_1$ if the point $x_1$ does not 
disconnect $\gamma$ into two connected components, such that the closure of one is not a loop $\eta$ such that $\eta-\{x_1\} \subset \CC-K$ (see Figure 6).

\begin{figure}[ht]
\centering
\resizebox{6cm}{!}{\includegraphics{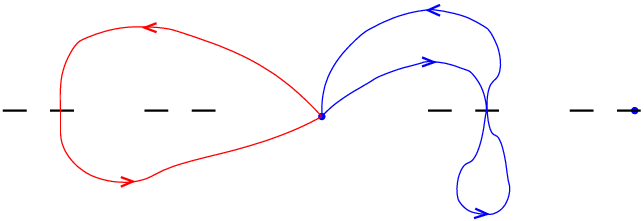}}    
 \put (-115,50) {\scriptsize $\gamma$}
  \put (-85,48) {\scriptsize $\eta$}
\caption{An essential loop $\gamma$ and a non-essential one $\eta$.} 
\end{figure}

We observe that if $\gamma_1$ and $\gamma_2$ 
are two homotopic paths in $X$ with fixed end-points, $\gamma_1\sim \gamma_2$, and with essential intersections with $G_1$, then 
$$
\gamma_1 \cap G_1 =\gamma_2 \cap G_1 
$$
because all points of $G_1$ are inner points of the Cantor set $K$ that are ``squeezed'' between a left and right sequences of points in $P$.
For non-essential intersections at $x_1$, we could locally slide homotopically the path to leave the intersection at $x_1$.  
Hence, under these conditions, $\gamma_1$ and $\gamma_2$ are homotopic in $X$ if and only if they are homotopic relative to $G_1$. 
Any continuous path for the monogenic topology is also continuous or the regular topology, and Borel homotopic equivalence is equivalent to homotopic
equivalence for the regular topology.

We can define in general:

\begin{definition}
Let $X$ be a path connected topological space and $z_0\in X$ a base point. We define
$$
\tilde X=\{[\gamma] ; \gamma : [0,1]\to X \text{ continuous}, \, \gamma(0)=z_0 \}
$$
where $[\gamma]$ denotes the homotopy class of $\gamma$ fixing its end-points.

We endow $\tilde X$ with the topology generated by the open sets
$$
U_{[\gamma]} =\{ [\gamma.\eta]; \eta ([0,1]) \subset U \text{ loop with } \eta(0)=\eta(1)= \gamma(1) \} 
$$
where $U$ is an open set of $X$ and $\gamma$ is a path in $X$.

\end{definition}

The topological space constructed does not depend on the base point for a path connected space $X$. If $z'_0\in X$ is another base point, then 
a path $\gamma$ from $z_0$ to $z'_0$ defines a homeomorphism depending only on the homotopy class of $\gamma$, $h_{[\gamma]}: \tilde X(z_0)\to \tilde X(z'_0)$,
by
$$
h_{[\gamma]} ([\eta])=[\eta.\gamma]
$$

The topological space $\tilde X$ is path connected, and even locally path-connected by construction. 
We will see shortly that it is simply connected.

We use this definition for $X=\CC-P$ endowed with the richer monogenic topology. By the precedent remarks, we don't really need 
Borel's topology to define \textit{the set} $\tilde X$. But the topology defined 
by taking a bases with the $(U_{[\gamma]})$ when $U$ runs over monogenic open sets endows $\tilde X$ with a richer topology.

Observe that since the monogenic topology is generated by the log-Euclidean metric on $\CC-P$, we could have taken as basis of neighborhoods 
$$
V_{[\gamma]} (\eps)=\{ [\gamma.\eta]; \eta \text{ loop with } \eta(0)=\eta(1)= \gamma(1), {\hbox{\rm diam}} (\eta) < \eps \} 
$$
where ${\hbox{\rm diam}} (\eta)$ is the log-Euclidean diameter. This comes down to request that $d_H(\gamma, \gamma.\eta) <\eps$ for the 
Hausdorff distance.

\begin{proposition}
The projection mapping $\pi : \tilde X \to X$ defined by
$$
\pi([\gamma]) =\gamma (1)
$$
is a continuous map and has the path-lifting property, and unique path-lifting property.
\end{proposition}

The path lifting and unique path-lifting property follows almost tautologically from the construction of $\tilde X$.

\begin{corollary}
The space $\tilde X$ is simply connected, i.e. any loop is contractible. 
\end{corollary}

This result is independent of the base point, hence we can take a loop starting and ending at $z_0$ and carry out the same proof as in \cite{Ha} p.65.
The homotopy to the null loop is obtained by using the parametrization.

The projection  $\pi$ makes of $\tilde X$ a $X$-space in the sense of Bourbaki (\cite{Bou}, p.1). In the non-semilocally simply connected situation the projection 
$\pi$ is not in general a local homeomorphism. In general, we have:

\begin{proposition}
For an open set $U\subset X$, and $[\gamma]\in \tilde X$, with $\gamma(1)\in U$, we have that $\pi: U_{[\gamma]} \to U$ is continuous and surjective.  
\end{proposition}

In our particular situation, if $U\cap G_1\not= \emptyset$ then  we have that $\pi: V\to U$ has infinite degree (even uncountable, as we have an 
uncountable number of pre-images generated by the uncountable Hawaian earrings). We say 
that $\pi$ has local infinite degree at $\tilde x_1\in p^{-1}(\{x_1\})$.

We are particularly interested in our particular situation in the case when $X$ is a Gromov length space.

\begin{proposition}
If $X$ is a Gromov length space, then  $\tilde X$ is also a Gromov length space and $\pi$ is a contraction.
\end{proposition}

\begin{proof}
The path-lifting property allows to pull-back the Gromov length distance into a Gromov length distance on $\tilde X$.
\end{proof}

Also, as in the classical situation, the group $\pi_1(X,z_0)$ acts on $\tilde X$. For a loop $\gamma$, starting and ending at $z_0$, the map 
$h_{[\gamma]}$ is a homeomorphism of $\tilde X$ that respects the fibers of $\pi$, so $\pi_1(X,z_0)$ acts on fibers. When we endow 
$\pi_1(X,z_0)$ with Bourbaki's admissible topology (\cite{Bou}, III, p.315) it becomes a topological group and this action is continuous 
and faithful on the fiber $\pi^{-1}(z_0)$.

Observe also that the fibers of $\pi$ are all homeomorphic. For another point $z_1\in X$, consider a path $\gamma$ from the base point $z_0$ to 
$z_1$. Then $h_{[\gamma]}$ gives an homeomorphism from $\pi^{-1}(z_0)$ to $\pi^{-1}(z_1)$.

We observe that in our particular situation, $\pi_1(X,z_0)$ is not generated by a countable bases, as show the circles in a local uncountable Hawaian earrings mentioned 
before.
 
In the classical theory, for $X$ semi-locally simply connected, to each subgroup $H$ of $\pi_1(X,z_0)$ it corresponds a quotient $X_H=\tilde X/\mathord\sim$ and 
a base point $\tilde z_0 \in X_H$,
such that for the resulting covering $\pi_H: X_H\to X$, $(\pi_H)_*(\pi_1(X_H,\tilde z_0))=H$. We have the same type of result in our general situation:

\begin{proposition}\label{prop:subgroup}
Let $X$ be path connected and locally path-connected. For every subgroup $H\subset \pi_1(X,z_0)$ there is an equivalence relation $\sim$ 
on $\tilde X$ respecting fibers of $\pi$, such that for the  quotient $X_H=\tilde X/\mathord\sim$ and the resulting projection $\pi_H:X_H\to X$, there is a base 
point $\tilde z_0\in X_H$, $\pi_H (\tilde z_{0})=z_0$, such that
$$
\left (\pi_{H}\right )_*(\pi_1(X_H,\tilde z_0))=H
$$
\end{proposition}

\begin{proof}
The proof is the same as Proposition 1.36 in \cite{Ha}. For $[\gamma],  [\gamma'] \in \tilde X$, we define $[\gamma]\sim [\gamma']$ if and only if
$\gamma(1)=\gamma'(1)$ and  $[\gamma.(\gamma')^{-1}] \in H$. This is an equivalence relation because $H$ is a subgroup. Note also that if 
$\gamma(1)=\gamma'(1)$ we have $[\gamma]\sim [\gamma']$ if and only if $[\gamma.\eta]\sim [\gamma'.\eta]$ for any loop $\eta$ with $\eta(0)=\eta(1)=\gamma(1)=\gamma'(1)$.
So, if two points in $U_{[\gamma]}$ and $U_{[\gamma']}$ are identified, then the whole neighborhoods $U_{[\gamma]}$ and $U_{[\gamma']}$ are identified by 
the equivalence relation and the open sets $U_{[\gamma]}/\mathord\sim$ are a bases for the topology of the quotient $X_H$. The equivalence relation respects fibers 
of $\pi$ since if $[\gamma]\sim [\gamma']$ then $\gamma(1)=\gamma'(1)$, thus $\pi_H:X_H\to X$ defined by $\pi_H(\overline{[\gamma]}) =\pi ([\gamma])$ 
is a well defined continuous projection. 
Let $\tilde z_0 \in \tilde X$ be the point corresponding to the class $[z_0]$ of the constant path in $X$ equal to the base point $z_0$. Choose for the 
base point in $X_H$, the image of $\tilde z_0$ by the quotient (that we still denote in the same way) . Now, the 
image $\left (\pi_{H}\right )_*(\pi_1(X_H,\tilde z_0))$ is exactly $H$, because a loop
$\gamma$ in $X$ based at $z_0$ lifts to $\tilde X$ into a path starting at $[z_0]$ and ending at $[\gamma]$, thus the image of this lifted path in $X_H$ is a 
loop if and only if $[\gamma]\sim [z_0]$, or $[\gamma]\in H$.
\end{proof}

\subsection{Construction of the Cantor Riemannium.}

After these preliminaries of monogenic algebraic topology, the construction of the 
Cantor Riemannium is straightforward.
We consider $X=\CC-P$ and we set $z_0=0$ as base point.
We have different relevant subgroups of $\pi_1(X,0)=\pi_1(X)$.

We have a representation of $\pi_1(X)$ in the group of translation of $\CC$, $\rho: \pi_1(X) \to \CC$ given by 
$$
[\gamma] \mapsto \int_{\gamma} F_{\nu} (z) \, dz 
$$
The image $B_\nu=\rho(\pi_1(X))$ is an additive  subgroup of $\RR$ which represents the possible values of the different branches at $0$, i.e.
$$
B_\nu=\BMon_{z=0} F_{\nu}
$$

Let $H_0= \Ker \rho$ be the kernel of this representation. We have that 
$H_0 \triangleleft \pi_1(X)$ is a normal closed subgroup of $\pi_1(X)$.
Hence, the equivalence relation associated to $H_0$ is closed and the quotient $X_{H_0}$ is separated. 

\begin{definition}
The Cantor Riemannium $\check C_{\nu}$ is the quotient space $X_{H_0}=\tilde X/\mathord\sim$ corresponding 
to the subgroup $H_0$. 
\end{definition}

By construction we get the important caracterization of the Cantor Riemannium:

\begin{theorem}
The Cantor Riemannium  $\check C_{\nu}$ is the minimal space 
where the Borel extension of $F_{\nu}$ is finite and single-valued.
\end{theorem}

Now we explore how to get other quotient spaces by taking other subgroups and identifie their quotient.
We consider the subgroup $H_1=\pi_1^r(X) \subset \pi_1(X)$ to be the subgroup 
of loops in $\CC-P$ which have an essential intersection with $G_1$ (we add the null loop also to 
have a subgroup). This subgroup is closed. 

\begin{proposition}
The quotient space $X_{H_1}=\tilde X/\mathord\sim$ corresponding to the subgroup $H_1$ is homeomorphic to $\CC-P\approx \cS^*_{\nu}$. 
\end{proposition}

Now, let $H_2 \subset \pi_1(X)$ be the subgroup of loops in $\CC-P$ which have an essential 
finite non-empty intersection with $G_1$ (we add the null loop also to 
have a subgroup). We check that we have a subgroup, and that this subgroup is not closed. 

\begin{proposition}
The quotient space $X_{H_2}=\tilde X/\mathord\sim$ corresponding to the subgroup $H_2$ is the 
restricted Cantor Riemannium $\tilde C^r_{\nu}$. 
\end{proposition}

Now let $H_3 \subset \pi_1(X)$ be the subgroup of loops in $\CC-P$ which have 
an essential finite intersection with $G_1$ (that can be empty and 
we add the null loop also to  have a subgroup). This subgroup is not closed.

\begin{proposition}
The quotient space $X_{H_3}=\tilde X/\mathord\sim$ corresponding to the subgroup $H_3$ is the restricted Cantor Riemannium $C^r_{\nu}$. 
\end{proposition}

Another interesting closed subgroup is $H_4\subset \pi_1(X)$, the subgroup of loops contained in $\CC-K$. The quotient space $X_{H_4}$ is in some sense 
the ``universal cover'' of the non-regular part.

%

\section{Borel extension of a holomorphic germ.}\label{sec:Borel_ext}

Our aim in this section is not to define the most general notion of Borel type extension, 
but just one that is has the uniqueness properties 
for the Cantor Riemannium construction.

We consider a holomorphic germ $F: U\to \CC$ defined in a neigborhood $U$ of $z_0\in \CC$ and a continuous path $\gamma: [0,1]\to \CC$ without self-intersections such that 
\begin{align*}
\gamma(0) &= z_0 \\
\gamma(1) &= z_1 \\
\end{align*}

\begin{definition}
Let $\epsilon >0$. A continuous path $\eta: [0,1]\to \CC$ without self-intersections such that 
\begin{align*}
\eta(0) &= z_0 \\
\eta(1) &= z_1 \\
\end{align*}
is $\epsilon$-Jordan homotopic to $\gamma$ if $\eta([0,1])$ and $\gamma([0,1])$  are $\epsilon$-close in Hausdorff topology, 
and $\eta([0,1])\cup \gamma([0,1])$ is a Jordan domain.
\end{definition}

\begin{definition}[Polar measure]
Let $R\in \CC(z)$ be a rational function. The polar measure of $R$ in the Riemann sphere $\overline {\CC}$ is 
the atomic complex measure
$$
\mu_R =\sum_{\rho} (\Res_\rho R) \ \delta_\rho
$$
i.e. the purely atomic measure  with support at the poles with complex atomic mass the residue at this point. 
\end{definition}

\begin{definition}
For $M>0$, we consider the space $\cM(M)$ of rational functions with a polar measure of total variation bounded by $M>0$,
$$
|\mu_R| (\CC) \leq M<+\infty \ .
$$
\end{definition}

\begin{definition}
Let $\gamma$ be a path  as before.
An asymptotically holomorphic  sequence $(f_n)$ on $\gamma$ is a sequence of meromorphic 
functions in $\cM(M)$ for some $M>0$, with simple poles out of $\gamma([0,1])$, 
such that for any $\epsilon>0$ and neighborhood $V_\epsilon (\gamma([0,1]))$
we have
$$
\lim_{n\to +\infty} \mu_{f_n} (V_\epsilon (\gamma([0,1])))=0 
$$
where $V_\epsilon (\gamma([0,1]))$ is the $\epsilon$-neighborhood of the support of $\gamma$.
\end{definition}

\begin{definition}\label{def:Borel_ext}
Let $F:U\to \CC$ be an  holomorphic germ and a path $\gamma$ as before.
If we have an asymptotically holomorphic sequence $(f_n)$ on $\gamma$ such that $(f_n)$ converges uniformly on $\gamma([0,1])$
to a continuous function $f:\gamma([0,1])\to \CC$ such that $f_{/U} =F'$, then we define 
$$
F(z_1)=\int_\gamma f(t) \, dt = \lim_{n\to +\infty} \int_\gamma f_n(t) \, dt  
$$
to be the Borel extension of $F$ at $z_1$ along $\gamma$ for the approximating sequence $(f_n)$.
\end{definition}

\begin{proposition}
If $(\gamma_k)$ is a sequence of $\epsilon_k$-Jordan homotopic paths to $\gamma$ with $\epsilon_k\to 0$, then 
$$
\lim_{n,k\to +\infty} \int_{\gamma_k} f_n(t) \, dt =\int_\gamma f(t) \, dt
$$
\end{proposition}

\begin{proof}
Let $\Omega_k$ be the Jordan domain with boundary $\gamma([0,1])\cup \gamma_k([0,1])$. We have by the Monogenic Residue Formula, 
Proposition \ref{prop:mon_residue}, when $k\to +\infty$,
$$
\int_{\gamma_k} f_n(t)(t) \, dt - \int_{\gamma_k} f_n(t)(t) \, dt = \mu_{f_n}(\Omega_k) \to 0
$$
and the result follows.
\end{proof}

The definition in general does depend on the choice of the asymptotically 
holomorphic sequence. In the case where $F'$ is the  Cauchy transform
of a measure $\mu$ supported on a Cantor set $K$,  and we approximate by a sequence $(f_n)$ of 
Cauchy transforms then we have the following Theorem:

\begin{proposition}
We assume that 
$$
F'(z)=f_{\mu}(z)=\frac{1}{2\pi i} \int_{\CC} \frac{d\mu(t)}{z-t}
$$
i.e. $\bar \partial \partial F =\mu$ in the sense of distributions is a finite total mass measure supported on $K=\supp \mu$ that is a 
Cantor set. We also assume that $\mu (\gamma([0,1])) =0$.

If we take  an asymptotically holomorphic sequence $(f_n)$  in Definition \ref{def:Borel_ext} such that 
$F(z_1) \not= \infty$, then the extension $F(z_1)$
is independent of the chosen sequence.

More generally, we can take a sequence $(f_n)$ of Cauchy transforms of measures $(\mu_{f_n})\subset \cM(M)$, for some $M>0$, such that
$\mu_{f_n}\to \mu$ and $\supp(\mu_{f_n})\to \supp(\mu)$.
\end{proposition}

\begin{lemma}
Given a Cantor set $K\subset \CC$ and $\gamma$ a continuous non-self-intersecting path, for any $\epsilon>0$ we can find an $\epsilon$-Jordan 
homotopic path such that it does not intersect $K$ except eventually at the end-points.
\end{lemma}

\begin{proof}
We use the useful result that homeomorphisms of the plane act transitively on Cantor subsets (see \cite{Mo} Theorem 7 p.93). Note that this is false in 
dimension $3$ as the Antoine necklace $A\subset \RR^3$ cannot be transported a Cantor set $C$ on a segment since $\pi_1(\RR^3-A)$ is not trivial and $\pi_1(\RR^3-C)$
is trivial.

Therefore, we can find a global homeomorphism $h:\CC\to \CC$ that transforms $K$ into the standard triadic Cantor set on $C\subset [0,1]\subset \CC$. Then the result is easy to prove, 
we first perturb $h(\gamma)$ into a Jordan equivalent path that only intersects the real line at isolated points, and then we do a second perturbation locally at  
each such point that intersects $C$ so that the new path does not intersect $C=h(K)$. The final observation is that Jordan small perturbation can be transported by 
the homeomorphism $h$.
\end{proof}

\begin{proof}
Using the previous Lemma we are reduced to the case when $\gamma$ intersects $K$ only at its end-point.
Consider another such sequence $(g_n)$. We have 
$$
\mu_{f_n-g_n} = \mu_{f_n}-\mu_{g_n} 
$$
thus $\mu_{f_n-g_n}\to 0$. By Lebesgue dominated convergence, the Cauchy transforms $f_{\mu_{f_n}}$ and $g_{\mu_{g_n}}$ converges pointwise to $F'$ on $\CC-K$. 
Since they are uniformly bounded on compact subsets of $\CC-K$, they also converge uniformly on compact sets of $\CC-K$ to $F'$. Again by Lebesgue dominated convergence on the 
integral along $\gamma$ (we can replace $\gamma$ by a Jordan homotopic rectifiable path) we have
$$
F(z_1)=\lim_{n\to +\infty} \int_\gamma f_n(t) \, dt=\lim_{n\to +\infty} \int_\gamma g_n(t) \, dt =\int_\gamma F'(t) \, dt
$$
\end{proof}

When we consider the Borel extension along a path $\gamma$ with self-intersections, but intersecting $C$ at isolated points transversally to $\RR$ where $F_{\nu_1}$ is finite, 
the relative homotopy class $[\gamma]_{\gamma\cap \RR}$ (the intersections points being fixed) can be decomposed into 
the sum of homotopy classes of simple loops without self-intersections with the same intersections on $C$.

\textbf{Acknowledgements.} I am grateful to K. Biswas, Y. Levagnini and M. Pe Pereira for remarks and corrections.

{\footnotesize{\textit{First draft in Chist\'en, 24th February 2020, and finished on the 34th day of confinement in Pozuelo de Alarc\'on.} }}

\end{document}